\documentclass[12pt,reqno]{amsart}
\usepackage{hyperref}
\usepackage{amssymb,amsthm}
\usepackage{amsfonts,cmtiup,comment,stmaryrd}

\usepackage{mathrsfs,cmtiup}

\makeatletter
\def\blfootnote{\xdef\@thefnmark{}\@footnotetext}
\makeatother

\newtheorem{theorem}{Theorem}[section]
\newtheorem{lemma}[theorem]{Lemma}
\newtheorem{proposition}[theorem]{Proposition}
\newtheorem{corollary}[theorem]{Corollary}

\theoremstyle{definition}

\newtheorem*{definition*}{Definition}
\newcommand{\ed}{\end{document}}
\newcommand{\F}{\mathbb{F}}
\newcommand{\N}{\mathbb{N}}
\newcommand{\Q}{\mathbb{Q}}
\newcommand{\f}{\varphi}
\newcommand{\g}{\gamma}

\renewcommand{\geq}{\geqslant}
\renewcommand{\leq}{\leqslant}

\let\leq=\leqslant
\let\geq=\geqslant
\numberwithin{equation}{section}

\begin{document}

\title{Compact groups in which commutators have finite right Engel sinks}

\author{Evgeny Khukhro}
\address{E. I. Khukhro: Charlotte Scott Research Centre for Algebra, University of Lincoln, U.K.}
\email{khukhro@yahoo.co.uk}

\author{Pavel Shumyatsky}

\address{P. Shumyatsky: Department of Mathematics, University of Brasilia, DF~70910-900, Brazil}
\email{pavel@unb.br}

\thanks{The second author was supported by  FAPDF and CNPq.} 

\begin{abstract} A right Engel sink of an element $g$ of a group $G$ is a subset containing all sufficiently long commutators $[...[[g,x],x],\dots ,x]$. We prove that if $G$ is a compact group in which, for some $k$, every commutator $[...[g_1,g_2],\dots ,g_k]$ has a finite right Engel sink, then $G$ has a locally nilpotent open subgroup. If in addition, for some positive integer $m$, every commutator $[...[g_1,g_2],\dots ,g_k]$ has a right Engel sink of cardinality at most $m$, then $G$ has a locally nilpotent subgroup of finite index bounded in terms of $m$ only.
\end{abstract}

\maketitle

\section{Introduction} We use the simple-commutator notation for left-normed commutators $[a_1,a_2,a_3,\dots, a_k]=[...[[a_1,a_2],a_3,\dots ,a_k]$ and the abbreviated notation for Engel words $[a,{}_nb]=[a,b,b,\dots ,b]$, where $b$ is repeated $n$ times. Any commutator of the form $[a_1,a_2,\dots, a_k]$ is called a \emph{$\g_k$-value} for short.

A \emph{right Engel sink} $R$ of an element $g$ of a group $G$ is a
set containing all sufficiently long commutators $[g,{}_nx]$, that is, for every $x\in G$  there is a positive integer $r(x,g)$ such that $[g,{}_n x]\in  R$ for any $n\geq r(x,g)$.

 We prove that if $G$ is a compact group in which, for some $k$, every $\g_k$-value has a finite right Engel sink, then $G$ has a locally nilpotent open subgroup. (A group is said to be locally nilpotent if every finite subset generates a nilpotent subgroup.) If in addition, for some positive integer $m$, every $\g_k$-value  has a right Engel sink of cardinality at most $m$, then $G$ has a locally nilpotent subgroup of finite index bounded in terms of $m$ only. First the latter  quantitative result is proved for finite groups in Theorem~\ref{t-finite}.  Then it is used in the proof of Theorem~\ref{t-profinite} on profinite groups. Finally Theorem~\ref{t-compact} on compact groups is proved using the profinite case.
 
Groups with restrictions on Engel sinks generalize Engel groups. A group $G$ is called an Engel group if for every $x,g\in G$ the equation $[x, {}_ng]=1$ holds for some $n=n(x,g)$.  Clearly, any locally nilpotent group is an Engel group. The converse is not true in general, as first shown by Golod's examples \cite{gol}, but has been proved in certain classes of groups. Using Zelmanov's deep results on Engel Lie algebras \cite{ze92,ze95,ze17}, Wilson and Zelmanov \cite{wi-ze} proved that any profinite Engel group is locally nilpotent. Later Medvedev \cite{med} extended this result to Engel compact (Hausdorff) groups.

We considered finite, profinite, and compact groups with restrictions on left or right Engel sinks of all elements in \cite{khu-shu153, khu-shu171, khu-shu162, khu-shu172, khu-shu191, khu-shu201}. In particular,  we proved in \cite[Theorem~3.1]{khu-shu172} that if all elements of a finite group $G$ have right Engel sinks of cardinality at most $m$, then $G$ has a normal subgroup of $m$-bounded order with nilpotent quotient; this also implies the existence of a normal nilpotent subgroup of $m$-bounded index.  (We write throughout  ``$m$-bounded'' to  abbreviate ``bounded above in terms of $m$ only''). In this paper, we impose the same condition only on the $\g_k$-values in a finite group~$G$.

\begin{theorem}\label{t-finite}
Let $k$ and $m$ be positive integers. If every $\g_k$-value in a finite group $G$ has a right Engel sink of cardinality at most $m$, then $G$ has a normal nilpotent subgroup of $m$-bounded index.
\end{theorem}

In contrast to the aforementioned \cite[Theorem~3.1]{khu-shu172}, under the hypotheses of Theorem~\ref{t-finite} for $k\geq 2$ one cannot obtain a normal subgroup of $m$-bounded order with nilpotent quotient:  examples are provided by the semidirect products $T\langle\alpha\rangle$, where $T$ is an elementary abelian $3$-group and $\alpha$ is an automorphism of $T$ acting by taking inverses.

For a compact group $G$ we proved in \cite[Theorem~1.1]{khu-shu172} that if all elements have finite right Engel sinks, then $G$ has a finite normal subgroup with locally nilpotent quotient; this also implies the existence of an open  normal locally nilpotent subgroup. In this paper, we impose the same condition only on the $\g_k$-values in a profinite or compact group~$G$.

\begin{theorem}\label{t-compact}
Let $k$ be a positive integer. If every $\g_k$-value in a compact group $G$ has a finite  right Engel sink, then $G$ has an open normal locally nilpotent subgroup.
\end{theorem}

In contrast to the aforementioned \cite[Theorem~1.1]{khu-shu172},  under the hypotheses of Theorem~\ref{t-compact} for $k\geq 2$ one cannot obtain a finite normal subgroup with locally nilpotent quotient, even in the case of profinite groups. An example is provided by a semidirect product $T\langle\alpha\rangle$, where $T$ is a Cartesian power of a group of order $3$  and $\alpha$ is an automorphism of $T$ acting by taking inverses; here the basis of neighbourhoods of identity in $T$ consists of subgroups with finite support.

It is also worth stating the following result on compact groups, in which the sizes of right Engel sinks are uniformly bounded. This corollary for profinite groups immediately follows from Theorem~\ref{t-finite} for finite groups, while the general case of compact groups is a consequence of  the proof of Theorem~\ref{t-compact}.

\begin{corollary}\label{c-em}
Let $k$ and $m$ be positive integers. If every $\g_k$-value in a compact group $G$ has a finite  right Engel sink of cardinality at most $m$, then $G$ has an open normal locally nilpotent subgroup of index bounded in terms of $m$ only.
\end{corollary}

\section{Preliminaries}

\subsection*{\bf Automorphisms.} If a group $A$ acts by automorphisms on a group $B$ we use the usual notation for commutators $[b,a]=b^{-1}b^a$ and commutator subgroups $[B,A]=\langle [b,a]\mid b\in B,\;a\in A\rangle$, as well as for centralizers $C_B(A)=\{b\in B\mid b^a=b \text{ for all }a\in A\}$ and $C_A(B)=\{a\in A\mid b^a=b \text{ for all }b\in B\}$. The automorphism induced by an automorphism $\varphi  $ on the quotient by a normal $\varphi $-invariant subgroup is denoted by the same letter $\varphi $.

 If a group $A$ acts by automorphisms on an elementary abelian $p$-group  $E$, then $E$ can be regarded as an $\F_pA$-module, in which the vector addition is the group multiplication in $E$ and the scalar multiplication by $i\in \F_p=\{0,1,2,\dots,p-1\}$ is taking the  $i$-th power. Similarly, if a group $A$ acts by automorphisms on a divisible abelian group  $D$, then $D$ can be regarded as an $\Q A$-module.

 An automorphism $\varphi$ of a finite group $G$ is called a coprime automorphism if  the orders of $\varphi$ and $G$ are coprime: $(|G|,|\varphi|)=1$. A coprime automorphism of a $p$-group is called a $p'$-automorphism. The following lemma collects several well-known properties of coprime automorphisms of finite groups, which will be often used without special references.

\begin{lemma}\label{l-nakr}
Let $A$ be a group of coprime automorphisms of a finite group $G$.
\begin{itemize}
  \item[\rm (a)]  If $N$ is a normal $A$-invariant subgroup of~$G$, then the fixed points of $A$ in  the quotient $G/N$ are covered by the fixed points of $A$ in $G$, that is, $C_{G/N}(A) = C_G(A)N/N$.

  \item[\rm (b)] We have $[[G, A],A]=[G,A]$. 

  \item[\rm (c)]   If $G$ is abelian, then $G=[G,A]\times C_G(A)$.

    \item[\rm (d)]   If $G$ is a $p$-group, then $A$ acts faithfully on the Frattini quotient $G/\Phi (G)$.
        
        \item[\rm (e)] If $G$ is an abelian $p$-group,  then $A$ acts faithfully on
        the subgroup  $\Omega_1(G)=\{x\in G\mid x^{p}=1\}$. 
\end{itemize}
  \end{lemma}

See, for example,  \cite[Theorem~6.2.2]{gor} for (a), \cite[Exercise 6.3]{gor} for (b),   \cite[Exercise 5.2.3]{gor} for (c),  \cite[Theorem~5.3.5]{gor} for (d), and  \cite[Theorem~5.2.4]{gor} for (e).

We also reproduce the so-called generalized Maschke's theorem.

\begin{lemma}\label{l-maschke}
Suppose that $A$ is a group acting by coprime automorphisms on a finite abelian group $E$. If $E=E_1\times E_2$, where $E_1$ is an $A$-invariant subgroup, then there is also an $A$-invariant subgroup $E_2^*$ such that $E=E_1\times E_2^*$.
\end{lemma}

See \cite[Satz~I.17.6]{hup}.

\subsection*{\bf Right Engel sinks.} It is easy to see that if an element $g$ of a group $G$ has a finite right Engel sink, then it has the smallest right Engel sink: if $R_1$ is a subset containing all commutators $[g,{}_n x]$ for any $n\geq r_1(x,g)$, and $R_2$ has the same property for a function $r_2(g,x)$,  then $ R_1\cap R_2$ also satisfies the same condition with the function $r(x,g)=\max\{r_1(x,g),r_2(x,g)\}$. We denote by ${\mathscr R}(g)$ the smallest right  Engel sink of $g$ when it exists. 

Clearly, if an element $g$ of a group $G$ has a finite right Engel sink ${\mathscr R}(g)$, then the image of $g$  in any quotient group  of any subgroup containing $g$ also has a finite right Engel sink, and the size of this right Engel sink does not exceed $|{\mathscr R}(g)|$. Since the $\g_k$-values in any quotient of any subgroup of a group $G$ are images of $\g_k$-values in $G$, the condition that all $\g_k$-values in $G$ have finite (or finite of size at most $m$) right Engel sinks is inherited by any quotient of any subgroup. We shall use this property throughout the paper without special references.

Recall that an element $g$ of a group $G$ is right Engel if for any $x\in G$ we  have $[g,{}_mx]=1$ for some $m=m(g,x)$ (this is equivalent to ${\mathscr R}(g)=\{1\}$). Recall that an element $g$ of a group $G$ is left  Engel if for any $x\in G$ we  have  $[x,{}_ng]=1$ for some $n=n(g,x)$.

The following lemma was proved by Heineken \cite{hei}.

\begin{lemma}[{\cite[12.3.1]{rob}}]\label{l-hei}
 If $g$ is a right Engel element of a group $G$, then $g^{-1}$ is a left Engel element.
\end{lemma}

The following elementary lemma in \cite{khu-shu172} characterises finite right Engel sinks.

\begin{lemma}[{\cite[Lemma~2.2]{khu-shu172}}]\label{l-r-sink}
If an element $g$ of a group $G$ has a finite right Engel sink, then
\begin{equation}\label{r-sink}
{\mathscr R}(g)=\{z\mid z=[g,{}_nx]=[g,{}_{n+m}x],\quad \; x\in G,\;\; n\geq 1,\;\;m\geq 1\}.
\end{equation}
\end{lemma}
Here, elements $x$ and numbers $m,n$ in \eqref{r-sink} vary for different $z$ and are not unique.

The next lemma is also similar to arguments in \cite{khu-shu172}, but we prove it here for the benefit of the reader.

\begin{lemma}\label{l3}
Suppose that a cyclic group $\langle a\rangle$ acts by 
automorphisms on a finite abelian group $V$ such that $V=[V,a]$. Then
\begin{itemize}
  \item[\rm (a)]  $C_V(a)=1$ and  the mapping $u\mapsto [u,a]$ is an automorphism of $V$;
  \item[\rm (b)] every element of $V$ is a $\g_k$-value in $V\langle a\rangle$  for any positive integer $k$;
  \item[\rm (c)] the minimal right Engel sink of an element $v\in V$ is the orbit of $v$ under the automorphism $u\mapsto [u,a]$, so if $|\mathscr{R}(v)|=s$
      then $$\mathscr{R}(v)=\big\{v=[v,{}_sa], \, [v,a], [v,{}_2a],\dots, [v,{}_{s-1}a]\big\}.$$
\end{itemize}
\end{lemma}

\begin{proof}
 Since $V$  is abelian, we have $[uv,a]=[u,a][v,a]$ for $u,v\in V$, so the mapping $u\mapsto [u,a]$ is a homomorphism with kernel $C_V(a)$. The condition $V=[V,a]$ means that the mapping $u\mapsto [u,a]$ is surjective, and therefore also  injective, since $V$ is finite, so that $C_V(a)=1$.  Part (a) is proved.

Thus, every $v\in V$ is a commutator of the form $v=[u,a]$, $u\in V$. Then also $u=[u_1,a]$, so $v=[u_1,a,a]$, and so on, which proves part (b).

It is now clear that the orbit of a non-trivial element $v\in V$ under the automorphism $u\mapsto [u,a]$  is the minimal right Engel sink $\mathscr{R}(v)$ of $v$, which proves (c).
\end{proof}

We shall also need another elementary lemma from~\cite{khu-shu172}

\begin{lemma}[{\cite[Lemma~2.3]{khu-shu172}}]\label{l-c-s}
Suppose that an element $g$ of a group $G$ has a finite right Engel sink of cardinality $|{\mathscr R}(g)|=m$. If $h\in C_G(g)$, then $h^{m!}$ centralizes
${\mathscr R}(g)$.
\end{lemma}

We proved the following theorem in \cite{khu-shu172}.

\begin{theorem}[{\cite[Theorem~3.1]{khu-shu172}}]\label{t-old-finite}
Let $G$ be a finite group, and $m$ a positive integer. Suppose that for every $g\in G$ the cardinality of the right Engel sink ${\mathscr R}(g)$ is at most $m$.
 Then $G$ has a normal subgroup $N$ of order bounded in terms of $m$ such that $G/N$ is nilpotent.
\end{theorem}

\subsection*{Ranks}
The (Pr\"ufer) rank of a finite group $G$ is the least positive integer $r$ such that every subgroup of $G$ can be generated by $r$ elements. The following lemma appeared independently and simultaneously in the papers of Yu.~M.~Gorchakov~\cite{grc}, Yu.~I.~Merzlyakov~\cite{me}, and as ``P.~Hall's lemma" in the paper of J.~Roseblade~\cite{rs}.

\begin{lemma}\label{l-gmh}
 Let $p$ be a~prime number. The rank of a~$p$-group of
automorphisms of an abelian finite $p$-group of rank~$r$ is bounded in
terms of~$r$.
\end{lemma}

The following well-known result is due to A.~Lubotzky and A.~Mann \cite[Propositions~2.6 and 4.2.6]{LM}.

\begin{lemma}\label{l-e-r}
Let $p$ be a prime, and $G$ a finite group of exponent $p^k$ and of rank $r$. Then there is a number $s(k,r)$ depending only on $k$ and $r$ such that $|G|\leq p^{s(k,r)}$.
\end{lemma}

\section{Finite groups}\label{s-finite}

We begin with considering direct products of finite simple groups.

\begin{lemma}\label{l1}
Let $k$ be a positive integer.
If a finite group $G$ is a direct product of non-abelian simple groups, then every element of $G$ is a $\g_k$-value.
\end{lemma}

\begin{proof}
  By the Ore conjecture proved in \cite{ore}, every element of a finite simple non-abelian group is a commutator.
   The same is true for any element $(g_1,\dots ,g_s)$ of a direct product $G=S_1\times\dots \times S_s$ of  non-abelian simple groups $S_i$: if $g_i=[a_i,b_i]$ for $a_i,b_i\in S_i$, then $(g_1,\dots ,g_s)=\big[(a_1,\dots ,a_s),\,(b_1,\dots ,b_s)\big]$. Then for any $g\in G$ we have $g=[a,b]$, where we can substitute $a=[c,d]$ for some $c,d\in G$, so $g=[c,d,b]$, and so on, next substituting  $c=[e,f]$, so $g=[e,f,d,b]$, etc.
\end{proof}

\begin{lemma}\label{l2}
Let $k$ be a positive integer.
Suppose that a finite group $G$ is a direct product of non-abelian simple groups. If all $\g_k$-values in $G$ have right Engel sinks of cardinality at most $m$, then the order of $G$ is $m$-bounded.
\end{lemma}

\begin{proof}
By Lemma~\ref{l1} the hypothesis means that every element of $G$ has a right Engel sink of cardinality at most $m$. Then $G$ has $m$-bounded order by Theorem~\ref{t-old-finite}.
\end{proof}

\begin{proof}[Proof of Theorem~\ref{t-finite}]
Recall that $G$ is a finite group in which every $\g_k$-value has a right Engel sink of cardinality at most $m$. We need to prove that the quotient of $G$ by the Fitting subgroup has $m$-bounded order.

First we perform reduction to soluble groups.

\begin{lemma}\label{l-sol}
The quotient $G/S(G)$ by the soluble radical $S(G)$ has $m$-bounded order.
\end{lemma}

\begin{proof}
In $G/S(G)$, the socle $E$ is a direct product of non-abelian finite simple groups. By Lemma~\ref{l2} the order of $E$ is $m$-bounded. Since the centralizer of $E$ in $G/S(G)$ is trivial, the group $G/S(G)$ embeds into the automorphism group of $E$ and therefore also has $m$-bounded order.
\end{proof}

Thus, we can assume from the outset that the group $G$ is soluble. Let $F_2(G)$ be the second term of the Fitting series, the inverse image of the Fitting subgroup $F(G/F(G))$ of the quotient by the Fitting subgroup. Since $F_2(G)/F(G)$ contains its centralizer in $G/F(G)$, it suffices to prove that the order of $F_2(G)/F(G)$ is $m$-bounded. Therefore we can assume that $G=F_2(G)$, that is, that $G/F(G)$ is nilpotent.

Let $P$ be a Sylow $p$-subgroup  of $G/F(G)$. It is sufficient to prove that $|P|$ is $m$-bounded, for any $p$, because this would also mean that any prime divisor of $|G/F(G)|$ is $m$-bounded, so that $G/F(G)$ has $m$-bounded number of non-trivial Sylow subgroups. Since $F(G)$ is the Fitting subgroup of the inverse image of $P$, we can assume that $G/F(G)=P$. Furthermore, since $P$ acts faithfully on the Hall $p'$-subgroup of $F(G)$, we can pass to the quotient by the Sylow $p$-subgroup of $F(G)$ and assume that $F(G)$ is a $p'$-group and $P$ is a subgroup of $G$. By
Gasch\"utz's theorem \cite[Satz III.4.2]{hup}, the image of $F(G)$ in the quotient of
$G$ by the Frattini subgroup $\Phi(F(G))$ of $F(G)$ is the Fitting subgroup of the quotient $G/\Phi(F(G))$. Therefore we can also assume that $\Phi(F(G))=1$, so that $F(G)$ is a direct product of elementary abelian $q$-groups over various $q$. In particular, $P$ acts by coprime automorphisms on $F(G)$.

A bound in terms of $m$ for the order of $P$ will follow from bounds in terms of $m$ for the exponent and rank of $P$ by Lemma~\ref{l-e-r}.

\begin{lemma}\label{l-exp}
The exponent of $P$ is $m$-bounded.
\end{lemma}

\begin{proof}
Let $g\in P$ be an arbitrary element  of $P$, and let $|g|=p^k$. Let $Q$ be a $q$-subgroup of $F(G)$ on which $g^{p^{k-1}}$ acts non-trivially; recall that $Q$ is elementary abelian. Let $V=[Q, g^{p^{k-1}}]$; then $C_V(g^{p^{k-1}})=1$. Then  for any $g^i\ne 1$ we also have $C_V(g^{i})=1$, so that $V=[V,g^i]$.

We now regard $V$ as a right $\F_q\langle g\rangle$-module. By Lemma~\ref{l3}(a)  the linear transformation $g-1$  is an automorphism of $V$, and its orbits are the minimal right Engel sinks of elements of $V$ by Lemma~\ref{l3}(c). The sizes of these sinks are at most $m$, since  every element of $V$ is a $\g_k$-value by Lemma~\ref{l3}(b). We obtain that  $(g-1)^{m!}=1$. In other words, the linear transformation $g$ of $V$ satisfies the polynomial $(X-1)^{m!}-1$.

We extend the ground field of $V$ to the algebraic closure $\overline{\mathbb F}_q$, naturally considering $\widetilde V=V\otimes _{\F_q}\overline{\mathbb F}_q$ as an $\overline{\mathbb F}_q\langle g\rangle$-module. Since the order of $g$ is coprime to the characteristic  of the field, $g$ is diagonalizable  as a linear transformation of $\widetilde V$.

This linear transformation also satisfies the polynomial $(X-1)^{m!}-1$. If $\alpha$ is an eigenvalue of $g$, then $(\alpha -1)^{m!}=1$. Applying the same arguments  to the
powers of $g$ we obtain $(\alpha ^i-1)^{m!}=1$ for all $i$.
This means that all the values $\alpha ^i-1$ are $m!$-th roots of unity in $\overline{\mathbb F}_q$. Therefore there are at most $m!$ different values of $\alpha ^i$. Hence the multiplicative order of $\alpha$ is at most $m!$. This is true for any eigenvalue of  $g$; hence, $|g|\leq m!$. Thus,  the exponent of $P$ is at most $m!$.
\end{proof}

\begin{lemma}  \label{l-r}
The rank of $P$ is $m$-bounded.
\end{lemma}
\begin{proof}
Let $A$ be a maximal normal abelian subgroup of $P$. Since $C_P(A)=A$, the quotient $P/A$ embeds in the automorphism group of $A$. The rank of $P/A$ is bounded in terms of the rank of $A$ by Lemma~\ref{l-gmh}. Therefore it suffices to obtain a bound for  the rank of $A$. 

Since $A$ acts coprimely on the abelian group $F(G)$, the latter is a direct product of minimal $A$-invariant subgroups by the generalized Maschke theorem (Lemma~\ref{l-maschke}). Choose minimal possible number of minimal $A$-invariant subgroups $V_1, \dots ,V_s$  of $F(G)$ such that $A$ acts faithfully on $V_1\times \dots \times V_s$ . Each $V_i$ can be regarded as an irreducible $\F_{q_i}A$-module for various (not necessarily different) primes $q_i$. The quotients $A/C_A(V_i)$ are cyclic because $A$ is abelian. It is sufficient to prove that $s$ is $m$-bounded, since $A=A/\bigcap_i C_A(V_i)$ embeds in the direct product of the $s$ cyclic groups $A/C_A(V_i)$.

By the minimality of $s$, for each $i$ the intersection $\bigcap_{j\ne i} C_A(V_j)$ is non-trivial, and we can choose an element $a_i\in \bigcap_{j\ne i} C_A(V_j)$ of prime order $p$. Then $[V_i,a_i]\ne 1$, and all elements of  $[V_i,a_i]$ are $\g_k$-values in $V_i\langle a_i\rangle$ by Lemma~\ref{l3}(b). Choose any non-trivial $v_i\in [V_i,a_i]$ for every $i=1,\dots ,s$. By Lemma~\ref{l3}(c) all commutators $[v_i,{}_na_i]$ are non-trivial elements of $\mathscr R(v_i)$ for all $n$.

It is easy to see that the product $w=v_1\cdots v_s$ is also a $\g_k$-value in $G$. Clearly, $[w,{}_na_i]=[v_i,{}_na_i]\in V_i$ for any $i$ and $n$. Therefore all these elements belong to $\mathscr R(w)$. Since $|\mathscr R(w)|\leq m$, it follows that, in particular, $s\leq m$, and the lemma is proved.
\end{proof}

By Lemma~\ref{l-e-r} the bounds in terms of $m$  for the exponent and rank of $P$ obtained in Lemmas~\ref{l-exp} and \ref{l-r} imply that the order of $P$ is $m$-bounded, which completes the proof of Theorem~\ref{t-finite}.
\end{proof}

\section{Profinite groups}\label{s-pf}
Recall that profinite groups are topological Hausdorff groups. Unless stated otherwise, a subgroup of a topological group will always mean a closed subgroup, all homomorphisms will be continuous, and quotients will be by closed normal subgroups. The same adage applies to taking commutator subgroups, normal closures, subgroups generated by subsets, etc., so, say, a subgroup generated by a subset $X$ will mean a subgroup generated by $X$ as a topological group. We refer the reader to the books \cite{wil} and \cite{ri-za} for the theory of profinite groups.

In this section we prove Theorem~\ref{t-compact} for profinite groups. 

\begin{theorem}\label{t-profinite}
Let $k$ be a positive integer. If every $\g_k$-value in a profinite group $G$ has a finite  right Engel sink, then $G$ has an open normal locally nilpotent subgroup.
\end{theorem}

 Recall that pronilpotent groups are defined to be inverse limits of finite nilpotent groups. Every profinite group has the largest normal pronilpotent subgroup, which is a closed subgroup being the intersection of all centralizers of all chief factors of all finite quotients. It is natural to call the largest normal pronilpotent subgroup  the \emph{pronilpotent radical}. The next lemma states that pronilpotent groups satisfying the hypothesis of Theorem~\ref{t-profinite} are locally nilpotent. Therefore the conclusion of the theorem for profinite groups can be stated as finiteness of the index of the pronilpotent radical, which will coincide with the Hirsch--Plotkin radical (which will therefore  be shown to be closed in this case).

\begin{lemma}\label{l-p-n}
Let $k$ be a positive integer.
If $G$ is a pronilpotent group in which all $\g_k$-values have finite right Engel sinks, then $G$ is locally nilpotent.
\end{lemma}

\begin{proof}
For any $\g_k$-value $x$, since ${\mathscr R}(x)$ is finite, there is
an open normal subgroup $N$ with nilpotent quotient $G/N$ such that ${\mathscr R}(x )\cap N=\{1\}$. On the other hand, ${\mathscr R}(x )\subset N$, since $G/N$ is nilpotent. Thus, ${\mathscr R}(x )=\{1\}$ for every $\g_k$-value $x$, which means that $x $ is a right Engel element. It immediately follows that every element  $g\in G$ is right Engel: for any $y\in G$ the commutator $[g,{}_{k-1}y]$ is a  $\g_k$-value, so that $1=[[g,{}_{k-1}y],{}_my]=[g,{}_{k+m-1}y]$ for some $m$. This means that $G$ is a  profinite Engel group. Then $G$ is locally nilpotent by the Wilson--Zelmanov theorem \cite[Theorem~5]{wi-ze}.
\end{proof}

By Lemma~\ref{l-p-n}, in the proof of Theorem~\ref{t-profinite} we only need to show that the pronilpotent radical has finite index. The first step is to prove that the quotient by the pronilpotent radical has an open nilpotent subgroup.

\begin{lemma}\label{l-pf1}
Let $k$ be a positive integer.
If $G$ is a profinite group in which all $\g_k$-values have finite right Engel sinks, then the quotient $G/K$ by the pronilpotent radical $K$ has an open nilpotent subgroup.
\end{lemma}

\begin{proof}
For every $\g_k$-value $x$ we choose an open normal subgroup $N_x $ such that ${\mathscr R}(x )\cap N_x =\{1\}$. Then $x $ is a right Engel element in $N_x \langle x \rangle$, and therefore, $x ^{-1}$ is a left Engel element in $N_x \langle x \rangle$ by Lemma~\ref{l-hei}. By Baer's theorem \cite[Satz~III.6.15]{hup}, in every finite quotient of $N_x \langle x \rangle$ the image of $x ^{-1}$ belongs to the Fitting subgroup. As a result, the subgroup $[N_x , x ^{-1}]=[N_x , x ]$ is pro\-nil\-po\-tent.

Let $\widetilde N_x $ be the normal closure of $[N_x , x ]$ in $G$. Since $[N_x , x ]$ is normal in the subgroup $N_x $ of finite index, $[N_x , x ]$ has only finitely many conjugates, so $\widetilde N_x $ is a product of finitely many normal subgroups of $N_x $, each of which is pro\-nil\-po\-tent.
Hence,  $\widetilde N_x $ is pro\-nil\-po\-tent. Therefore all the subgroups $\widetilde N_x $ are contained in the pro\-nil\-po\-tent radical~$K$.

The image of every $\g_k$-value $x$ in the quotient $G/K$ is centralized by the image of $N_x $, which has finite index in $G$. Hence every $\g_k$-value in the quotient $G/K$ has only finitely many conjugates. A~profinite group with this property has an open nilpotent subgroup by a result of Detomi, Morigi, and Shumyatsky \cite[Proposition~3.4]{det-mor-shu}.
\end{proof}

\begin{proof}[Proof of Theorem~\ref{t-profinite}] Recall that $G$ is a profinite group in which  $\g_k$-values have finite right Engel sinks, and we need to show that $G$ has an open locally nilpotent normal subgroup. By Lemma~\ref{l-p-n} this is equivalent to showing that the quotient $G/K$ by the pronilpotent radical $K$   is finite.

Throughout what follows we can assume that $k\geq 2$, since for $k=1$ an even stronger result was proved in \cite[Theorem~1.1]{khu-shu172}.

Since $G/K$ has an open nil\-po\-tent subgroup  index by Lemma~\ref{l-pf1}, we can assume from the outset that $G/K$ is nil\-po\-tent. Let $A$ be a subgroup of $G$ such that $A/K$ is a maximal normal abelian subgroup of $G/K$ (which is automatically closed). Since $A/K$ contains its centralizer in $G/A$, the quotient $G/A$ embeds in the automorphism group of $A/K$. Therefore it is sufficient to show that $A/K$ is finite. The subgroup $K$ is also the pronilpotent radical of $A$. Therefore we can simply replace $G$ with $A$ and assume from the outset that $G/K$ is abelian.

Recall that both $K$ and $G/K$ are Cartesian products of their Sylow subgroups. Let $K=\prod_{i}Q_i$, where $Q_i$ is the Sylow $q_i$ subgroup of $K$. For any prime $p$, the Sylow $p$-subgroup of $G/K$ acts faithfully on the Hall $p'$-subgroup of $K$, which is the Cartesian product of the pro-$q_i$ groups $Q_i$ for $q_i\ne p$. An element of the Sylow  $p$-subgroup of $G/K$ acting nontrivially on $Q_i$ also acts nontrivially on the Frattini quotient $Q_i/\Phi (Q_i)$; this follows from the same property in Lemma~\ref{l-nakr}(d) for the finite quotients of $G$. Therefore the image of $K$ in the quotient $G/\prod_{i}\Phi (Q_i)$ by the Cartesian product of the Frattini subgroups is the pronilpotent radical of this quotient. Therefore we can assume that $\prod_{i}\Phi (Q_i)=1$, so that $K$ is an abelian group.

Our aim now is reduction to the case where the cardinalities of right Engel sinks of $\g_k$-values are uniformly bounded. This will enable us to finish the proof by applying Theorem~\ref{t-finite} on finite groups. First we introduce closed subsets of the direct product  of $k$ copies of $G$ with the product topology.

\begin{lemma}\label{l-closed}
Let $G$ be a profinite group in which all $\g_k$-values have finite right Engel sinks, and let $t$ be a positive integer. Then the  set
$$
R_{t}=\{(x_1,\dots,x_k)\in \underbrace{G\times \dots\times G}_k\mid |{\mathscr R}([x_1,\dots ,x_k])|\leq t\}
$$
is closed in $\underbrace{G\times \dots\times G}_k$ in the product topology.
\end{lemma}

\begin{proof} To lighten the notation, we set $G^{\times k}=\underbrace{G\times \dots\times G}_k$.  We wish to show equivalently that the complement of $R_k$ is an open subset of $G^{\times k}$. Every element $(g_1,\dots,g_k)\in (G^{\times k}\setminus R_k)$ is characterized by the fact that $|{\mathscr R}([g_1,\dots,g_k])|\geq t+1$. Let $z_1,z_2,\dots ,z_{t+1}$ be some $t+1$ distinct elements in ${\mathscr R}([g_1,\dots,g_k])$. Using Lemma~\ref{l-r-sink}  for every $z_i$, $i=1,\dots ,t+1$,   we can write for some $x_i\in G$ and some $m_i,n_i\in \N$
\begin{equation}\label{e-open}
z_i= [[g_1,\dots,g_k],{}_{m_i}x_i]=[[g_1,\dots,g_k],{}_{m_i+n_i}x_i], \quad \text{where }  m_i\geq 1,\;\;n_i\geq 1.
 \end{equation}
 Let $N$ be an open normal subgroup of $G$ such that the images of $z_1,z_2,\dots ,z_{t+1}$ are distinct elements in $G/N$.  For any $u_i\in N$, substituting elements $g_iu_i$ instead of $g_i$, $i=1,\dots ,k$, into equations \eqref{e-open} shows that the right Engel sink $ {\mathscr R}([g_1u_1,\dots,g_ku_k])$ contains an element in each of the $t+1$ cosets $z_iN$. Indeed, $[[g_1u_1,\dots,g_ku_k],{}_{m_i}x_i]=z_iv_i$ for some $v_i\in N$ and $[z_i,{}_{n_i}x_i]=z_i$. Then $$
 [[[g_1u_1,\dots,g_ku_k],{}_{m_i}x_i],{}_{sn_i}x_i]=[z_iv_i, ,{}_{sn_i}x_i]\in z_iN\qquad \text{for any } s\in \N.$$

 Thus, all elements in the coset $(g_1N,\dots ,g_kN)$ are contained in $G^{\times k}\setminus R_k$. We have shown that every element of $G^{\times k}\setminus R_k$ has a neighbourhood that is also contained in $G^{\times k}\setminus R_k$, which is therefore an open subset of $G^{\times k}$.
\end{proof}

We return to the proof of Theorem~\ref{t-profinite}. The hypothesis  means that
$$
\underbrace{G\times \dots\times G}_k=\bigcup_{i}R_i,
$$
where the subsets $R_i$ are closed by Lemma~\ref{l-closed}. By the Baire category theorem \cite[Theorem~34]{kel}, one of these sets contains an open subset; that is, there is an open subgroup $U$ and cosets $a_1U,\dots, a_kU$ such that $(a_1U,\dots, a_kU)\subseteq R_{m}$ for some $m$. In other words,
\begin{equation}\label{e-cosets}
  | {\mathscr R}([a_1u_1,\dots,a_ku_k])|\leq m\qquad \text{for any } u_i\in U.
\end{equation}

We now prove a key lemma, which will enable us to use Theorem~\ref{t-finite} on  finite groups. Recall that we now assume that our group $G$ is metabelian and $k\geq 2$.

\begin{lemma}\label{l-key}
The cardinalities of the right Engel sinks of $\g_k$-values in $U$ are uniformly bounded in terms of $m$.
\end{lemma}

\begin{proof}
We will be using the standard commutator formulae $[ab,c]=[a,c][a,c,b][b,c]$ and $[a, bc]=[a,c][a,b,c][a,b]$ that hold in any group, as well as the formulae $[a,b,c,d]=[a,b,d,c]$ and $[[a,b]^{-1},c]=[a,b,c]^{-1}$, which hold in any metabelian group. We also use the fact that any commutators commute in $G$, since $G$ is metabelian.

Starting with the property \eqref{e-cosets}, we now consecutively obtain uniform bounds in terms of $m$ for the sizes of the right Engel sinks of the commutators
\begin{equation}\label{e-s}
  [u_1,\dots,u_s,a_{s+1}u_{s+1},\dots, a_ku_k]
\end{equation}
obtained from $[a_1u_1,\dots,a_ku_k]$ in \eqref{e-cosets} by dropping some of the first factors $a_i$.  More precisely, we use induction on $s$ to prove that the right Engel sink of the commutator \eqref{e-s} has cardinality at most $m^{3^s}$. Although we can regard property \eqref{e-cosets} as the basis of induction with $i=0$, we also show the calculation for $s=1$, as this case is a little different from the general induction step.

For $s=1$, in order to deal with  ${\mathscr R}([u_1,a_2u_2,\dots,a_ku_k])$ , we  write
$$[a_1u_1,a_2u_2]=[a_1,a_2u_2]\cdot [a_1,a_2u_2, u_1]\cdot [u_1,a_2u_2],$$ whence
\begin{align*}
[u_1,a_2u_2,\dots ,a_ku_k]={}&[a_1u_1,\dots,a_ku_k]\\
&{}\cdot [a_1,a_2u_2,\dots,a_ku_k]^{-1}\notag\\
&{}\cdot [a_1,a_2u_2, \dots ,a_ku_k,u_1]^{-1},\notag
\end{align*}
where in the third factor on the right we moved $u_1$ to the right end using the fact that $G$ is metabelian and $k\geq 2$.
Then for any $x\in G$ and any $n\in \N$ we have
\begin{align}\label{e-s12}
 [u_1,a_2u_2,\dots ,a_ku_k,\,{}_{n}x]={}&[a_1u_1,\dots,a_ku_k,\,{}_{n}x]\\
&{}\cdot [a_1,a_2u_2,\dots,a_ku_k,\,{}_{n}x]^{-1}\notag\\
&{}\cdot [a_1,a_2u_2, \dots ,a_ku_k,\,{}_{n}x,\,u_1]^{-1},\notag
\end{align}
where in the third factor on the right we interchanged ${}_{n}x$ and $u_1$ using the fact that $G$ is metabelian and $k\geq 2$. In view of Lemma~\ref{l-r-sink}, all elements of  ${\mathscr R}([u_1,a_2u_2,\dots,a_ku_k])$ appear as values of the left-hand side of \eqref{e-s12} for sufficiently large $n$, which can also be chosen as large as required. We can consider values of $n$  that are also large enough for the first two factors on the right to belong to the right Engel sinks of the initial subcommutators of length $k$, as well as for the initial subcommutator of length $k+n$ of the third factor to belong to the right Engel sink of the initial subcommutator of length $k$. These three initial subcommutators of length $k$ on the right have the form of a commutator in~\eqref{e-cosets}. Therefore by property \eqref{e-cosets} for large enough $n$ there are at most $m$ values for each of the three factors on the right of \eqref{e-s12}. As a result, $|{\mathscr R}([u_1,a_2u_2,\dots,a_ku_k])|\leq m^3$, as claimed.

In the induction step we suppose that
\begin{equation}\label{e-inhyp}
   | {\mathscr R}([u_1,\dots,u_s,a_{s+1}u_{s+1},\dots, a_ku_k]) |\leq m^{3^s}\qquad \text{for any } u_i\in U.
\end{equation}
We now write
\begin{align*}
  [u_1,\dots , u_s,a_{s+1}u_{s+1}]  =&[u_1,\dots , u_s,a_{s+1}] \\
   & {}\cdot [u_1,\dots , u_s,u_{s+1}]\\
   & {}\cdot [u_1,\dots , u_s,a_{s+1}, u_{s+1}],
\end{align*}
 whence
\begin{align*}
[u_1,\dots , u_s,u_{s+1},a_{s+2}u_{s+2},\dots, a_ku_k]={}& [u_1,\dots , u_s,a_{s+1}u_{s+1},\dots,a_ku_k]\\
&{}\cdot [u_1,\dots , u_s,a_{s+1},a_{s+2}u_{s+2},\dots,a_ku_k]^{-1}\notag\\
&{}\cdot [u_1,\dots , u_s,a_{s+1},a_{s+2}u_{s+2}, \dots ,a_ku_k,u_{s+1}]^{-1},\notag
\end{align*}
where in the third factor on the right we moved $u_{s+1}$ to the right end using the fact that $G$ is metabelian and $k\geq 2$.
Then for any $x\in G$ and any $n\in \N$ we have
\begin{align}\label{e-ss2}
 [u_1,\dots , u_s,u_{s+1},a_{s+2}u_{s+2},\dots, a_ku_k,\,{}_{n}x]={}& [u_1,\dots , u_s,a_{s+1}u_{s+1},\dots,a_ku_k,\,{}_{n}x]\\
&{}\cdot [u_1,\dots , u_s,a_{s+1},a_{s+2}u_{s+2},\dots,a_ku_k,\,{}_{n}x]^{-1}\notag\\
&{}\cdot [u_1,\dots , u_s,a_{s+1},a_{s+2}u_{s+2}, \dots ,a_ku_k,\,{}_{n}x,u_{s+1}]^{-1},\notag
\end{align}
where in the third factor on the right we interchanged ${}_{n}x$ and $u_{s+1}$ using the fact that $G$ is metabelian and $k\geq 2$. In view of Lemma~\ref{l-r-sink}, all elements of
$${\mathscr R}( [u_1,\dots , u_s,u_{s+1},a_{s+2}u_{s+2},\dots, a_ku_k])$$ appear as values of the left-hand side of \eqref{e-ss2} for sufficiently large $n$, which can also be chosen as large as required. We can consider values of $n$  that are also large enough for the first two factors on the right to belong to the right Engel sinks of the initial subcommutators of length $k$, as well as for the initial subcommutator of length $k+n$ of the third factor to belong to the right Engel sink of the initial subcommutator of length $k$. These three initial subcommutators of length $k$ on the right have the form of a commutator in~\eqref{e-inhyp}. By the induction hypothesis \eqref{e-inhyp} for large enough $n$ there are at most $m^{3^s}$ values for each of the three factors on the right of \eqref{e-ss2}. As a result,
$$|{\mathscr R}( [u_1,\dots , u_s,u_{s+1},a_{s+2}u_{s+2},\dots, a_ku_k])|\leq (m^{3^s})^3=m^{3^{s+1}},
$$
 which completes the proof for $s+1$.

For $s=k$ we obtain in the end that $|{\mathscr R}( [u_1,\dots , u_k])|\leq m^{3^k}$ for any $u_i\in U$.
\end{proof}

We can now finish the proof of Theorem~\ref{t-profinite}. By Lemma~\ref{l-key}, in every finite quotient of $U$  the right Engel sinks of $\g_k$-values have cardinalities uniformly bounded in terms of $m$. By Theorem~\ref{t-finite} each of these finite quotients has a normal nilpotent subgroup of $m$-bounded index. Hence $U$ has a pronilpotent normal subgroup of finite $m$-bounded index. Since $U$ has finite index in $G$, it follows that the pronilpotent radical of $G$ has finite index. By Lemma~\ref{l-p-n} the pronilpotent radical is locally nilpotent.
\end{proof}

\medskip

\section{Compact groups}
Recall that compact groups are topological Hausdorff groups, and
subgroups, homomorphisms, etc. are understood in the topological sense. When considering ordinary subgroups or group homomorphisms, regardless of the topology, we use the adjective ``abstract''. We refer the reader to the books \cite{h-m} and \cite{h-r} for the theory of compact groups.

In this section we prove Theorem~\ref{t-compact} about compact groups in which all $\g_k$-values have finite right Engel sinks. We use the structure theorems for compact groups and the results of the preceding section on profinite groups with this property.

Recall that a group $H$ is said to be \emph{divisible} if for every $h\in H$ and every positive integer $k$ there is an element $x\in H$ such that $x^k=h$.

\begin{lemma}\label{l-div}
Suppose that $H$ is a divisible group in which all $\g_k$-values have finite right Engel sinks. Then for any $g,x\in H$ there is a positive integer $n(x,g)$ such that \begin{equation*}\label{e-free}
  [[[g,{}_{k-1}x],{}_nx],\,[g,{}_{k-1}x]]=1\qquad \text{for all }n\geq n(x,g).
\end{equation*}
\end{lemma}

\begin{proof}
In view of $[g,{}_{k-1}x]$ being a $\g_k$-value, let $|{\mathscr R}([g,{}_{k-1}x])|=m$. Let $h\in H$ be an element such that $h^{m!}=[g,{}_{k-1}x]$. Since $h$ centralizes $[g,{}_{k-1}x]$, by Lemma~\ref{l-c-s} we obtain that $[g,{}_{k-1}x]=h^{m!}$ centralizes ${\mathscr R}([g,{}_{k-1}x])$. By the definition of ${\mathscr R}([g,{}_{k-1}x])$, there is a positive integer $n(x,g)$ such that $[[g,{}_{k-1}x],{}_nx]\in {\mathscr R}([g,{}_{k-1}x])$ for all $n\geq n(x,g)$. By the above, $[[[g,{}_{k-1}x],{}_nx],\,[g,{}_{k-1}x]]=1$ for all $n\geq n(x,g)$.
\end{proof}

\begin{proof}[Proof of Theorem~\ref{t-compact}] Recall that $G$ is a compact group in which all $\g_k$-values have finite right Engel sinks, and we need to prove that $G$ has an open locally nilpotent subgroup.

By the well-known structure theorems (see, for example, \cite[Theorems~9.24 and 9.35]{h-m}), the connected component $G_0$ of the identity  in $G$ is a closed divisible subgroup such that $G_0/Z(G_0)$ is a Cartesian product of simple compact Lie groups, while the quotient $G/G_0$ is a profinite group.

By Theorem~\ref{t-profinite} the profinite group $G/G_0$ has an open locally nilpotent subgroup, the inverse image of which is also an open subgroup of $G$. Therefore the proof will be complete with the following proposition. 

\begin{proposition}\label{p-comp}
Suppose that $G$ is a compact group in which all $\g_k$-values have finite right Engel sinks and the quotient $G/G_0$ by the connected component $G_0$ of the identity is locally nilpotent. Then $G$ is  locally nilpotent.
\end{proposition}

\begin{proof} First we prove that $G_0$ is abelian.

 \begin{lemma}\label{l-abelian}
 The connected component $G_0$ of the identity  in $G$ is an abelian group.
 \end{lemma}

 \begin{proof}
 We  need to show that $G_0/Z(G_0)=1$. Otherwise $G_0/Z(G_0)$ has a  subgroup $H$ isomorphic to a simple compact Lie group. As a linear group in characteristic zero, $H$ satisfies the Tits Alternative \cite{tits}, by which it either contains a non-abelian free subgroup or has a soluble subgroup of finite index. The latter is impossible for simple compact Lie groups, so $H$ contains an abstract  non-abelian free subgroup $F=\langle a, b\rangle$. Since $H$ is a divisible group, by Lemma~\ref{l-div} there is a positive integer $n=n(a,b)$ such that $  [[[a,{}_{k-1}b],{}_nb],[a,{}_{k-1}b]]=1$. But this  is obviously false for the free generators $a,b$ of~$F$, a contradiction.
\end{proof}

We proceed with the proof of Proposition~\ref{p-comp}.

Let  $T$ be the torsion part of the divisible abelian group $G_0$, so that $T$ is an abstract normal  subgroup of $G$. We can regard the abstract quotient $V=G_0/T$ as a vector space over $\Q$, on which the group $G/G_0$ acts by linear transformations.

\begin{lemma}\label{l-qspace}
All elements of $V=G_0/T$ are right Engel in the abstract quotient group $G/T$.
\end{lemma}

\begin{proof}
Arguing by contradiction, suppose that an element $u\in V$ is not right Engel in $G/T$. Then there is $x\in G$ such that $[u,{}_{s}x]\ne 1$ for any $s\in \N$; in particular, the $\g_k$-value $[u,{}_{k-1}x]$ is also not right Engel. By hypothesis, then $\mathscr R([u,{}_{k-1}x])$ is non-trivial and finite. By Lemma~\ref{l-r-sink} there is $y\in G/T$ such that
$$
1\ne [[u,{}_{k-1}x],{}_{m}y]=[[u,{}_{k-1}x],{}_{m+n}y]
$$
for some positive integers $m,n$. For brevity, let $v=[[u,{}_{k-1}x],{}_{m}y]$. Regarding $V$ as a right $\Q (G/T)$-module, we see that the span $U$ of the elements $v=[v,{}_ny],\, [v,y],\dots , [v,{}_{n-1}y]$ is a $y$-invariant finite-dimensional subspace of $V$.

In the notation of  a right $\Q (G/T)$-module, commutators $[w,g]$ for $w\in V$, $g\in G$ are written as $w(g-1)$. Thus, $U$ is spanned by the elements $v=v(y-1)^n,\,  v(y-1),\dots , v(y-\nobreak 1)^{n-1}$. Clearly, $(y-1)^n=1$, which means that $y-1$ is a nonsingular linear transformation of $U$ of finite (multiplicative) order dividing $n$.  Then $y-1$ is diagonalizable over the algebraic closure $\overline{\mathbb Q}$ of $\Q$ as a linear transformation of $\widehat{ U}=U\otimes _{\Q}\overline{\mathbb Q}$. The  eigenvalues of $y-1$ on~$\widehat{ U}$ are roots of unity. But then $y^2-1$ will have eigenvalues of infinite multiplicative order on~$\widehat{U}$. Indeed,  if $\lambda$ is any eigenvalue of $y$, then  $\lambda -1=\zeta$ is an eigenvalue of $y-1$, which is a root of 1. Then $\lambda^2 =\zeta^2+2\zeta +1$, whence $\lambda^2-1=\zeta(\zeta+2)$. Here,  $\zeta +2$ cannot be a root of 1, unless $\zeta=-1$, but then $\lambda =\zeta+1=0$, which is impossible, since $y$ is nonsingular. It is also impossible to have $\zeta +2=0$.   Hence $\lambda^2-1$ has infinite multiplicative order. Therefore  $y^2-1$ is a nonsingular linear transformation of infinite multiplicative order on $\widehat{U}$, and therefore also on $U$.  Since $U$ is finite-dimensional,  there is $w\in U$ such that the vectors $w(y^2-1)^j$ are all different and non-zero. In terms of the group operations, this means that all commutators $[w,{}_jy^2]$ are different, and therefore the $\g_k$-value  $[w,{}_{k-1}y^2]$ does not have a finite right Engel sink, contrary to the hypothesis of the proposition.
\end{proof}

\begin{lemma}\label{l-ln2}
The abstract quotient group $G/T$ is locally nilpotent.
\end{lemma}

\begin{proof}
  The abstract quotient $G/T$ is an Engel group: for any $x,y\in G$, firstly $[x,{}_ly]\in G_0$ for some $l\in \N$, because $G/G_0$ is locally nilpotent by hypothesis, and then $[[x,{}_ly],{}_my]\in T$ for some $m\in \N$ by Lemma~\ref{l-qspace}. Thus, $G/T$ is an Engel group, which is abelian-by-(locally nilpotent), and such groups are known to be locally nilpotent (see, for example, \cite[12.3.3]{rob}).
\end{proof}


\begin{lemma}\label{l-torsion}
The abstract torsion part $T$ of $G_0$ consists of right Engel elements in $G$.
\end{lemma}

\begin{proof}
Since $T$ is contained in the abelian subgroup $G_0$, the action of $G$ by conjugation on $G_0$ factors through to the action of the profinite group $G/G_0$.
Suppose the opposite, that there is an element $a\in T$ that is not right Engel in $G$. Then there is $h\in G/G_0$ such that $[a,{}_{s}h]\ne 1$ for any $s\in \N$; in particular, the $\g_k$-value $[a,{}_{k-1}h]$ is also not right Engel. By hypothesis, then $\mathscr R([a,{}_{k-1}h])$ is non-trivial and finite. By Lemma~\ref{l-r-sink} there is $g\in G/G_0$ such that
$$
1\ne [[a,{}_{k-1}h],{}_{m}g]=[[a,{}_{k-1}h],{}_{m+n}g]
$$
for some positive integers $m,n$.

As an abstract group, $T$ is a direct product of its Sylow subgroups, each of which is an abstract normal subgroup of $G$.  Therefore there is a prime $p$ such that for the projection $b$ of $a$ onto the Sylow $p$-subgroup of $T$ we also have
$$
1\ne [[b,{}_{k-1}h],{}_{m}g]=[[b,{}_{k-1}h],{}_{m+n}g]
$$
for the same positive integers $m,n$. For brevity, let $v=[[b,{}_{k-1}h],{}_{m}g]$.

Let $B$ be the abstract subgroup generated by the elements $v=[v,{}_ng],\, [v,g],\dots , [v,{}_{n-1}g]$. This is a finite $g$-invariant $p$-subgroup such that $B=[B,g]$. Let $\psi$ denote the automorphism of $B$ induced by conjugation by $g$, so that also $B=[B,\psi]$. Then 
$C_B(\psi )=1$ by Lemma~\ref{l3}(a).


Since $B$ is finite, $\psi$ has finite order. Let $\langle \psi\rangle=\langle \alpha\rangle\times \langle \f\rangle$, where $\langle \alpha\rangle$ is the (possibly trivial) Sylow $p$-subgroup of $\langle \psi\rangle$, and $\f$ is a $p'$-automorphism of order $n$ with $p\nmid n$.  We observe that also $C_B(\f)=1$, for otherwise the $p$-automorphism $\alpha $ would have a non-trivial fixed point in $C_B(\f)$, which would also be in $C_B(\psi )$.

Since $B$ is finite, the normalizer $N_G(B)$ is closed and contains the closed procyclic subgroup  $\langle g\rangle$. Let $g=g_pg_{p'}$, where $g_p$ is an element of the Sylow $p$-subgroup of $\langle g\rangle$ of $G/G_0$ and $g_{p'}$ is an element of the Cartesian product of the other Sylow $q$-subgroups of $\langle g\rangle$ over primes $q\ne p$. Since $B$ is finite, the centralizer $C_{\langle g\rangle}(B)$  is a closed subgroup of finite index in $\langle g\rangle$, and $\f$ is the image of $g_{p'}$ in the quotient $\langle g\rangle/C_{\langle g\rangle}(B)$.

Changing notation, we assume from the outset that $B$ is a non-trivial $g$-invariant finite abelian $p$-subgroup of $T$, the closed procyclic subgroup  $\langle g\rangle$  of $G/G_0$ has trivial Sylow $p$-subgroup, and $\f$ is a $p'$-automorphism of $B$ of order $n$ induced by conjugation by $g$ such that $C_B(\f)=1$.

Recall that the abstract Sylow $p$-subgroup of $T$ is a divisible $p$-group, which is a direct product of copies of the quasicyclic group $C_{p^\infty}$. Therefore all the roots of all elements of $B$ in $G_0$ form an abstract  $g$-invariant divisible  $p$-subgroup $D$. Note that $D$ is also normalized by $\langle g\rangle$. The centralizer of each of the finite subgroups $\Omega _i(D)=\{u\in D\mid u^{p^i}=1\}$, $i=1,2,\dots $, is a closed subgroup. Therefore the quotient $\langle g\rangle/C_{\langle g\rangle}(\Omega _i(D))$ is a finite $p'$-group, and $g$ induces by conjugation a $p'$-automorphism of $\Omega _i(D)$.
By Lemma~\ref{l-nakr}(e)  the automorphism induced by $g$ on each subgroup $\Omega _i(D)$ has the same order $n$ as  on $\Omega_1(D)=\Omega_1(B)$ and as on $B$. Therefore $g$ induces an automorphism of the same order $n$ on $D$;  we denote this automorphism of $D$ also by $\f$. Since $C_B(\f)=1$, we have $C_D(\f )=1$. Hence,   $\Omega _i(D)=\{x^{-1}x^\f\mid x\in \Omega _i(D)\}$ for every $i$ by Lemma~\ref{l3}(b), and therefore $D=\{x^{-1}x^\f\mid x\in D\}$. Then
\begin{equation}\label{e-split}
dd^\f\cdots d^{\f^{n-1}}=x^{-1}x^\f (x^{-1}x^\f)^\f\cdots (x^{-1}x^\f)^{\f^{n-1}}=1\qquad \text{for all }d\in D.
\end{equation}

Now let $\,\overline{\!  D}$ be the closure of  $D$,  which is contained in $G_0$.  The element $g$ normalizes $\,\overline{\!  D}$ and induces by conjugation an automorphism of the same order $n$ as on $D$, since  $D$ is contained in the closed subgroup $C_G(g^n)$, which therefore also contains $\,\overline{\!  D}$. We denote this automorphism of $\,\overline{\!  D}$  by the same letter $\f$. The subgroup $\,\overline{\!  D}$  is compact  but not profinite; therefore it cannot be periodic (see \cite[Theorem~28.20]{h-r}).  Therefore we have the non-trivial abstract quotient $\,\overline{\!  D}/T(\,\overline{\!  D})\ne 1$, where $T(\,\overline{\!  D})=\,\overline{\!  D}\cap T$ is the abstract torsion part of $\,\overline{\!  D}$.

Since $\f$ is a continuous    automorphism of $\,\overline{\!  D}$, equation~\eqref{e-split} implies that the same property is enjoyed by $\,\overline{\!  D}$:
\begin{equation}\label{e-split2}
  aa^\f\cdots a^{\f^{n-1}}=1\qquad \text{for all }a\in \,\overline{\!  D}.
\end{equation}
It follows that $\f$ has no non-trivial fixed points in the abstract quotient $\,\overline{\!  D}/T(\,\overline{\!  D})$. Indeed,  by substituting any $c\in C_{\,\overline{\!  D}/T(\,\overline{\!  D})}(\f)$  into~\eqref{e-split2} we obtain $c^n=1$, whence $c=1$, since $\,\overline{\!  D}/T(\,\overline{\!  D})$ is torsion-free. In particular, $\f$ is a non-trivial automorphism of finite order of  the abstract quotient $\,\overline{\!  D}/T(\,\overline{\!  D})$.

Since the abstract group $G/T$ is locally nilpotent by Lemma~\ref{l-ln2}, the semidirect product $(\,\overline{\!  D}/T(\,\overline{\!  D})\langle \f\rangle$ is also locally nilpotent. The element of finite order $\f$ belongs to the torsion part of this group, which has trivial intersection with the normal subgroup $\,\overline{\!  D}/T(\,\overline{\!  D})$. Hence $\f$ centralizes $\,\overline{\!  D}/T(\,\overline{\!  D})$. This contradiction completes the proof. 
\end{proof}

We can now finish the proof of the proposition. Lemmas~\ref{l-ln2} and \ref{l-torsion} imply that $G$ is an Engel group. Indeed, 
for any $x,y\in G$, firstly $[x,{}_ly]\in  T$ for some $l\in \N$, since $G/T$ is locally nilpotent by Lemma~\ref{l-ln2}. Then  $[[x,{}_ly],{}_my]=1$ for some $m\in \N$ by Lemma~\ref{l-torsion}. Thus, $G$ is an Engel abelian-by-(locally nilpotent) group, and such groups are known to be locally nilpotent (see, for example, \cite[12.3.3]{rob}).

Proposition~\ref{p-comp} is proved.
\end{proof}

As explained above, Proposition~\ref{p-comp} completes  the proof of Theorem~\ref{t-compact}.
\end{proof}

\begin{proof}[Proof of Corollary~\ref{c-em}.]
Recall that here $k$ and $m$ are positive integers and  $G$ is a compact group in which every $\g_k$-value has a finite right Engel sink of cardinality at most $m$, and we need to show that $G$ has an open locally nilpotent subgroup of index bounded in terms of $m$. When $G$ is profinite, then $G$ has a pronilpotent open subgroup of $m$-bounded index, since every finite quotient has a nilpotent subgroup of $m$-bounded index by Theorem~\ref{t-finite}. By Lemma~\ref{l-p-n} a pronilpotent subgroup is locally nilpotent. Thus, for a compact group $G$, the quotient $G/G_0$ by the connected component of the identity has an open locally nilpotent subgroup of $m$-bounded index. The full inverse image of this subgroup is locally nilpotent by Proposition~\ref{p-comp}.
\end{proof}

\ed